\documentclass[11pt]{amsart}
\usepackage[margin=1in]{geometry}
\usepackage[utf8]{inputenc}
\usepackage{amsthm, amssymb}
\usepackage[colorlinks=true, pdfstartview=FitV, linkcolor=blue, citecolor=blue, urlcolor=blue]{hyperref}
\usepackage[colorinlistoftodos]{todonotes}
\usepackage{algorithm}
\usepackage{algpseudocode}
\usepackage{pgf}
\usepackage{tikz}
\usepackage{tikz-cd}
\usetikzlibrary{positioning,shapes,shadows,arrows}
\usepackage{bbm,bm}
\usepackage{enumitem}
\usepackage{mathabx}
\usepackage{comment}
\usepackage{ytableau}
\usepackage{thmtools}
\usepackage{thm-restate}

\newtheorem{prop}{Proposition}[section]
\newtheorem{thm}[prop]{Theorem}

\newtheorem{lemma}[prop]{Lemma}
\newtheorem{cor}[prop]{Corollary}

\theoremstyle{remark}
\newtheorem{exa}[prop]{Example}
\newtheorem{rem}[prop]{Remark}
\newtheorem{defn}[prop]{Definition}

\newcommand{\NN}{\mathbb{N}}

\newcommand{\BPD}{\mathsf{BPD}}
\newcommand{\LTBPD}{\mathsf{LTBPD}}

\newcommand{\fG}{\mathfrak{G}}
\newcommand{\fGb}{\fG}
\newcommand{\fL}{\mathfrak{L}}
\newcommand{\fLb}{\fL}
\newcommand{\topLas}{\widehat{\fL}}
\newcommand{\topGro}{\widehat{\fG}}
\newcommand{\fS}{\mathfrak{S}}

\newcommand{\PT}{\mathsf{PerfectTab}_\downarrow}

\newcommand{\supp}{\mathsf{supp}}

\newcommand{\wt}{\mathsf{wt}}

\newcommand{\Q}{\mathbb{Q}}

\newcommand{\snow}{\mathsf{snow}}
\newcommand{\rajcode}{\mathsf{rajcode}}
\newcommand{\raj}{\mathsf{raj}}

\newcommand{\snowflake}{\Asterisk}

\newcommand{\invcode}{\mathsf{invcode}}

\newcommand{\std}{\mathsf{std}}

\newcommand{\btile}{
 \begin{tikzpicture}[x=1em,y=1em,thick,color = blue]
\draw[step=1,gray,thin] (0,0) grid (1,1);
\draw[color=black, thick, sharp corners] (0,0) rectangle (1,1);
\end{tikzpicture}}

\newcommand{\htile}{
\begin{tikzpicture}[x=1em,y=1em,thick,color = blue]
\draw[step=1,gray,thin] (0,0) grid (1,1);
\draw[color=black, thick, sharp corners] (0,0) rectangle (1,1);
\draw(1.0,0.5)--(0.0,0.5);
\end{tikzpicture}
}

\newcommand{\vtile}{
\begin{tikzpicture}[x=1em,y=1em,thick,color = blue]
\draw[step=1,gray,thin] (0,0) grid (1,1);
\draw[color=black, thick, sharp corners] (0,0) rectangle (1,1);
\draw(0.5,1.0)--(0.5,0.0);
\end{tikzpicture}}

\newcommand{\ptile}{
\begin{tikzpicture}[x=1em,y=1em,thick,color = blue]
\draw[step=1,gray,thin] (0,0) grid (1,1);
\draw[color=black, thick, sharp corners] (0,0) rectangle (1,1);
\draw(0.5,1.0)--(0.5,0.0);
\draw(1.0,0.5)--(0.0,0.5);
\end{tikzpicture}}

\newcommand{\rtile}{
\begin{tikzpicture}[x=1em,y=1em,thick,color = blue]
\draw[step=1,gray,thin] (0,0) grid (1,1);
\draw[color=black, thick, sharp corners] (0,0) rectangle (1,1);
\draw(1.0,0.5)--(0.5,0.5)--(0.5,0.0);
\end{tikzpicture}}

\newcommand{\jtile}{
\begin{tikzpicture}[x=1em,y=1em,thick,color = blue]
\draw[step=1,gray,thin] (0,0) grid (1,1);
\draw[color=black, thick, sharp corners] (0,0) rectangle (1,1);
\draw(0.5,1.0)--(0.5,0.5)--(0.0,0.5);
\end{tikzpicture}}

\definecolor{darkblue}{rgb}{0.0,0,0.7} 
\definecolor{darkred}{rgb}{0.7,0,0} 
\definecolor{darkgreen}{rgb}{0, .6, 0} 
\newcommand{\definition}[1]{{\color{darkred}\emph{#1}}} 

\title{Connection between Schubert polynomials
and top Lascoux polynomials}
\author[T.~Yu]{Tianyi Yu}
\address[T. Yu]{Department of Mathematics, UC San Diego, La Jolla, CA 92093, U.S.A.}
\email{tiy059@ucsd.edu}
\date{February 2023}

\begin{document}

\maketitle
\begin{abstract}

Schubert polynomials form a basis of the polynomial 
$\Q[x_1, x_2, \cdots]$.
This basis and its structure constants have received 
extensive study.
Recently, Pan and Yu initiated the study of top Lascoux polynomials. 
These polynomials form a basis
of the vector space $\widehat{V}$, a subalgebra of $\Q[x_1, x_2, \cdots]$
where each graded piece has finite dimension. 
This paper connects Schubert polynomials and top Lascoux polynomials via a simple operator.
We use this connection to show these two bases share the same 
structure constants.
We also translate several results on Schubert polynomials 
to top Lascoux polynomials, 
including combinatorial formulas
for their monomial expansions and supports.

\end{abstract}

\section{Introduction}

For a permutation $w$,
Lascoux and Sch\"utzenberger~\cite{LS:Schubert}
recursively defined the \definition{Schubert polynomial} 
$\fS_w$ using \definition{divided difference operators}.
These polynomials represent Schubert cycles in flag varieties and have been extensively investigated 
from various perspectives.  
We summarize some significant results on 
Schubert polynomials relevant to this paper.
\begin{enumerate}
\item The set of all Schubert polynomials
forms a basis of the polynomial ring 
$\mathbb{Q}[x_1, x_2, \cdots]$.
Products of Schubert polynomials can be 
expanded positively into Schubert polynomials 
(i.e. the expansion only involves positive 
integer coefficients):
$$
\fS_u \fS_v = \sum_{w} c_{u,v}^w \fS_w,
$$
The coefficient $c^{w}_{u,v}$ is known as the 
\definition{Schubert structure constant}. 
A major open problem in algebraic combinatorics
is to compute $c^{w}_{u,v}$ combinatorially.
\item Lam, Lee and Shimozono~\cite{LLS} introduced 
the \definition{(reduced) bumpless pipedreams 
(BPD)} to compute the monomial expansion of 
Schubert polynomials. 
\item The Schubert polynomial can be 
expanded positively into \definition{key polynomials}~\cite{RS}. 
\item 
The dual character of the flagged Weyl module
of a diagram $D$ is denoted as $\chi_D$.
The Schubert polynomial $\fS_w$ is $\chi_{RD(w)}$
where $RD(w)$ is the Rothe diagram of $w$~\cite{KP}.
\item 
Adve, Robichaux, and Yong~\cite{ARY}
introduced \definition{perfect tableaux}
to compute the support of Schubert polynomials.
\item The Schubert polynomials have 
the \definition{saturated Newton polytope (SNP)} 
property~\cite{FMS}.
\end{enumerate}

The key polynomials mentioned above
are denoted as $\kappa_\alpha$, 
where $\alpha$ is a \definition{weak composition}.
They are the characters of 
Demazure modules~\cite{De}.
Lascoux~\cite{Las} introduced an inhomogeneous analogue
of $\kappa_\alpha$ known as 
the \definition{Lascoux polynomial} $\fLb_\alpha$.
The lowest-degree terms of $\fLb_\alpha$ form
$\kappa_\alpha$.
Recently, Pan and Yu~\cite{PY2} introduced 
the \definition{top Lascoux polynomial} $\topLas_\alpha$
which consists of the 
highest-degree terms of $\fLb_\alpha$.
Let $\widehat{V}$ be the $\Q$-span
of all top Lascoux polynomials. 
Unlike the Schubert polynomials, 
the set of all top Lascoux polynomials
is not linearly independent.
To resolve this, Pan and Yu called a weak composition 
\definition{snowy} if its positive entries are distinct.
Then $\{\topLas_\alpha: \alpha \textrm{ is snowy}\}$
forms a basis of $\widehat{V}$.
By~\cite[Theorem 1.2]{PY2}, 
every top Lascoux polynomial is a top Lascoux indexed by snowy weak composition multiplied by an integer.
In the rest of this paper,  
we only focus on $\topLas_\alpha$
when $\alpha$ is snowy. 

Pan and Yu showed that $\widehat{V}$
is closed under multiplication.
Thus, $\widehat{V}$ can be viewed as a graded algebra where the grading is given by degrees 
of polynomials. 
Each graded piece of $\widehat{V}$
has finite dimension.
Pan and Yu computed the Hilbert series
of $\widehat{V}$.
Intuitively, $\widehat{V}$ is much smaller
than the polynomial ring $\Q[x_1, x_2, \cdots]$ which has no Hilbert series. 

Just like the Schubert polynomials,
$\topLas_\alpha$ can be defined 
recursively using
divided difference operators 
(see~(\ref{EQ: Define top Las})). 
This resemblance leads to a strong connection
between Schubert polynomials
and top Lascoux polynomials,
which is the main focus of this paper.

\begin{defn}
Define the following involution 
on polynomials in $\Q[x_1, \cdots, x_n]$
where each variable has degree at most $m$:
$$r_{m,n}(f) := (x_1\cdots x_n)^m f(x_n^{-1}, \cdots, x_1^{-1}).$$   
\end{defn}

In \S\ref{S: Relations}, we show that each top Lascoux polynomial 
can be realized as $r_{m,n}(\fS_w)$
for some $m, n, w$ and vice versa. 
Following this connection,
we translate the results on $\fS_w$ summarized above to $\topLas_\alpha$. 

\begin{enumerate}
\item 
Products of top Lascoux polynomials can be 
expanded positively into top Lascoux polynomials:
$$
\topLas_\alpha \topLas_\gamma
= \sum_{\sigma} d^\sigma_{\alpha, \gamma} \topLas_\sigma. 
$$
We call the coefficient
\definition{top Lascoux structure constants}.
Every $d^\sigma_{\alpha, \gamma}$ is $c^{w}_{u,v}$
for some permutations $u, v, w$
and vice versa 
(see \S\ref{S: structure constants}).
\item We give a monomial expansion of top Lascoux 
polynomials using (modified) bumpless-pipedreams
(see \S\ref{S: BPD}).
\item The top Lascoux polynomials can be 
expanded positively into key polynomials
(see \S\ref{S: key expansion}).
\item The top Lascoux polynomial $\topLas_\alpha = \chi_{\snow(D(\alpha))}$
where $\snow(D(\alpha))$ is some diagram defined in~\cite{PY2}(See \S\ref{S: character}).
\item The support of top Lascoux polynomials 
can be computed using perfect tableaux 
(see \S\ref{S: character}). 
\item The top Lascoux polynomials have the SNP property (See \S\ref{S: character}).
\end{enumerate}

These properties of the top Lascoux basis precisely mirror those of
the Schubert basis. 
It may be interesting to do Schubert calculus 
in the graded ring $\widehat{V}$
which as a defined and understood Hilbert series. 
By (1), computing the Schubert structure constants 
is the same as computing the top Lascoux structure constants.

Another potential application of our results
is to understand the Grothendieck polynomial $\fGb_w$, 
the $K$-theoretic analogue of $\fS_w$.
By Shimozono and Yu~\cite{SY},
$\fGb_w$ positively
into $\fLb_\alpha$.
Consequently, 
the top degree component of $\fGb_w$,
denoted as $\topGro_w$,
expands positively into $\topLas_\alpha$.
There has been a recent surge in the study of 
$\topGro_w$~\cite{CY, CY2, DMS, Haf, HMSS, PSW, PY, RRRSW, RRW}.
Together with the expansion of $\topGro_w$ into $\topLas_\alpha$,
one would translate our results on $\topLas_\alpha$ to $\topGro_w$.
In particular, 
when $w$ is vexillary ($2143$-avoiding),
Pechenik and Scrimshaw~\cite{PS} showed
that $\fGb_w$ is just a Lascoux polynomial.
We may then translate our results
to $\topGro_w$ for vexillary $w$.
See Remark~\ref{R: top Gro char} for one such potential application. 

The rest of the paper is structured as follows. 
In Section \S\ref{S: Background}, 
we provide an overview of the 
necessary background information.
In \S\ref{S: Relations},
we use $r_{m,n}$ to relate the Schubert polynomials
and top Lascoux polynomials. 
The subsequent sections explore 
various applications of this relationship. 
In Section \S\ref{S: structure constants}, 
we examine the connection between the structure 
coefficients of top Lascoux polynomials and 
Schubert polynomials. 
In \S\ref{S: BPD}, 
we derive a combinatorial formula 
for top Lascoux polynomials 
from the BPD formula of Schubert polynomials. 
In \S\ref{S: key expansion}, 
we translate the key expansion of Schubert polynomials 
to obtain a key expansion of top Lascoux polynomials. 
In \S\ref{S: character},
we show the top Lascoux polynomials 
are certain dual characters of the flagged Weyl modules and characterize the support of top Lascoux polynomials.

\section{Background}
\label{S: Background}

\subsection{Schubert polynomials}

Let $S_+$ be the group of permutations of 
$\{1, 2, \cdots\}$
where only finitely many elements are permuted.
The simple transpositions $s_1, s_2, \dots$ where $s_i = (i, i+1)$ generate $S_+$.
For any positive number $n$, $S_n$ is a subgroup of $S_+$ consisting of $w$ that only permutes $[n] = \{1, 2, \cdots, n\}$.
We represent $w \in S_+$ by its \definition{one-line notation} $[w(1), \dots, w(n)]$
for some $n$ large enough such that 
$w \in S_n$.

A \definition{weak composition} $\alpha = (\alpha_1, \alpha_2, \dots, )$ is an infinite sequence
of non-negative numbers with finitely many
positive entries. 
The \definition{support} of $\alpha$ 
is the set 
$\supp(\alpha) := \{i: \alpha_i > 0\}$.
We represent $\alpha$ as $(\alpha_1, \alpha_2, \dots, \alpha_n)$ 
where $\supp(\alpha) \subseteq [n]$. 
Let $x^\alpha := 
x_1^{\alpha_1}x_2^{\alpha_2}\cdots$
and $|\alpha| := \sum_{i\geqslant 1}\alpha_i$.

We say $(i,j)$ is an inversion of $w \in S_+$
if $i < j$ and $w(i) > w(j)$.
The \definition{inversion code} of $w$, 
denoted as $\invcode(w)$ is 
a weak composition defined as
$$\invcode(w)_i := |\{j: (i,j) \textrm{ is an inversion of }w\}|.$$

The \definition{Schubert polynomials} $\fS_w$ are indexed by permutations
from $S_+$.
When a weak composition is weakly decreasing, 
we say it is a \definition{partition}.
When $\invcode(w)$ is a partition, 
we say $w$ is a \definition{dominant permutation}.
Define the 
\definition{Newton divided difference operator}:
$$\partial_i(f) := \dfrac{f - s_i f}{x_i-x_{i+1}},$$
where $s_i f$ is the operator that swaps $x_i$
and $x_{i+1}$.
Now we can define the \definition{Schubert polynomial}
of $w \in S_+$ recursively~\cite{LS:Schubert}.

\begin{align*}
\fS_w = \begin{cases}
x^{\invcode(w)} & \text{if $w$ is dominant} \\
\partial_i (\fS_{w s_i} ) &\text{if $w(i) < w(i+1)$.}
\end{cases}
\end{align*}

The set of Schubert polynomials 
form a $\mathbb{Q}$-basis of the polynomial ring $\mathbb{Q}[x_1, x_2, \cdots]$.
For $u, v \in S_+$,
the product $\fS_u \fS_v$ can be expanded into Schubert polynomials.
Let $c^{w}_{u,v}$ be the coefficient of $\fS_w$ in this expansion.
By geometric results, 
$c^w_{u,v}$ is a non-negative integer known as the 
\definition{Schubert structure constants}.

\subsection{Key polynomials and top Lascoux polynomials}
The key polynomials $\kappa_\alpha$ 
are indexed by weak compositions. 
Lascoux and Sch{\"u}tzenberger~\cite{LS:Groth} define the key polynomials recursively.
using the operator
$\pi_i(f) := \partial_i(x_i f)$:
\begin{align*}
\kappa_\alpha := \begin{cases}
x^\alpha & \text{if $\alpha$ is a partition,} \\
\pi_i (\kappa_{s_i\alpha} ) &\text{if $\alpha_i<\alpha_{i+1}$,}
\end{cases}
\end{align*}
where $s_i$ swaps the $i^\textsuperscript{th}$ and $(i+1)^\textsuperscript{th}$ entries of $\alpha$. 

The \definition{top Lascoux polynomial} $\topLas_\alpha$ 
are homogeneous polynomials 
indexed by \definition{snowy} weak compositions: 
weak compositions whose positive entries are distinct.
Following~\cite[Lemma 4.23]{PY2}, 
we may define these polynomials recursively.
Define the operator $\widehat{\pi}_i$ 
as $$\widehat{\pi}_i(f) := \pi_i(x_{i+1}f)= x_i x_{i+1} \partial_i(f).$$
Then define
\begin{equation}
\label{EQ: Define top Las}
\topLas_\alpha := \begin{cases}
x^\alpha & \text{if $\alpha$ is a partition,} \\
\widehat{\pi}_i (\topLas_{s_i\alpha}) &\text{if $\alpha_i<\alpha_{i+1}$.}
\end{cases}
\end{equation}

By the study of Pan and Yu~\cite{PY2}, 
the vector space
\begin{equation}
\label{EQ: widehat V}
\widehat{V} := \mathbb{Q}\textrm{-}\textsf{span} \{ \topLas_\alpha: \alpha \textrm{ is a snowy weak composition}\}
\end{equation}
is an sub-algebra of $\Q[x_1, x_2, \dots]$.
Its basis is
given by the spanning set in 
(\ref{EQ: widehat V})
and its Hilbert series is
$\prod_{m > 0} \left( 1 + \frac{q^m}{1 - q} \right)$.
In particular, each graded piece of $\widehat{V}$ has finite dimension. 

For weak compositions $\alpha, \gamma,$ and $\delta$,
let $d^\delta_{\alpha, \gamma}$ be the coefficient of 
$\topLas_\delta$ in the expansion of 
$\topLas_\alpha \times \topLas_\gamma$.
We call them the 
\definition{top Lascoux structure constants}.
Later in \S\ref{S: structure constants},
we show
each $d^\delta_{\alpha, \gamma}$
is the Schubert structure constant $c^w_{u,v}$
for some permutations $u, v, w$
and vice versa. 

\subsection{Bumpless pipedreams}
The \definition{(reduced) bumpless pipedreams (BPD)}, introduced by Lam, Lee and Shimozono\cite{LLS},
are combinatorial
objects that give a monomial expansion of a Schubert polynomial. 
For permutation $w \in S_n$,
a BPD is an $n \times n$ grid built by the following six tiles: 
$$
\btile, \quad \rtile, \quad \jtile, \quad \ptile, \quad \htile, \quad \vtile.
$$
We adopt the convention that row $1$ is the topmost row
and column $1$ is the left most column. 
For each $i \in [n]$, 
we require a pipe to enter from the bottom of column $i$
and end at the rightmost edge of row $w(i)$.
Moreover, two pipes cannot cross more than once. 
\begin{exa}
\label{E: BPD}
There are three BPDs for the permutation in $S_4$
with one-line notation $[2,1,4,3]$.
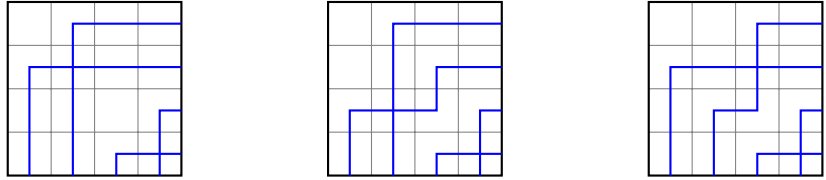
\[
\begin{tikzpicture}[x=1.5em,y=1.5em,thick,color = blue]
\draw[step=1,gray,thin] (0,0) grid (4,4);
\draw[color=black, thick, sharp corners] (0,0) rectangle (4,4);
\draw(.5, 0)--(.5,2.5)--(4,2.5);
\draw(1.5, 0)--(1.5,3.5)--(4,3.5);
\draw(2.5, 0)--(2.5,0.5)--(4,0.5);
\draw(3.5, 0)--(3.5,1.5)--(4,1.5);
\end{tikzpicture}
\quad\quad\quad\quad\quad
\begin{tikzpicture}[x=1.5em,y=1.5em,thick,color = blue]
\draw[step=1,gray,thin] (0,0) grid (4,4);
\draw[color=black, thick, sharp corners] (0,0) rectangle (4,4);
\draw(.5, 0)--(.5,1.5)--(2.5,1.5)--(2.5,2.5)--(4,2.5);
\draw(1.5, 0)--(1.5,3.5)--(4,3.5);
\draw(2.5, 0)--(2.5,0.5)--(4,0.5);
\draw(3.5, 0)--(3.5,1.5)--(4,1.5);
\end{tikzpicture}
\quad\quad\quad\quad\quad
\begin{tikzpicture}[x=1.5em,y=1.5em,thick,color = blue]
\draw[step=1,gray,thin] (0,0) grid (4,4);
\draw[color=black, thick, sharp corners] (0,0) rectangle (4,4);
\draw(.5, 0)--(.5,2.5)--(4,2.5);
\draw(1.5, 0)--(1.5,1.5)--(2.5,1.5)--(2.5,3.5)--(4,3.5);
\draw(2.5, 0)--(2.5,0.5)--(4,0.5);
\draw(3.5, 0)--(3.5,1.5)--(4,1.5);
\end{tikzpicture}
\]
\end{exa}

We let $\BPD(w)$ be the set of BPDs of a permutation $w$.
We call $\btile$ a \definition{blank}.  
The \definition{blank-weight} of a BPD $D$
is a weak composition where the $i^\textsuperscript{th}$
entry counts the number of $\btile$ in row $i$.
We denote it as $\wt_{\Box}(D)$ 
to emphasize that the weight comes from the blanks.
Then BPD gives a combinatorial formula for Schubert polynomials.
\begin{thm}[\cite{LLS}]
For a permutation $w \in S_n$,
$$
\fS_w = \sum_{D \in \BPD(w)} x^{\wt_\square(D)}.
$$
\end{thm}
For instance, by Example~\ref{E: BPD},
when $w$ has one-line notation $[2,1,4,3]$,
$\fS_{w} = x_1x_3 + x_1 x_2 + x_1^2$.

\subsection{Diagrams}
A \definition{diagram} is a finite subset 
of $\NN \times \NN$.
We may represent a diagram by putting a cell at row $r$
and column $c$ for each $(r,c)$ in the diagram.
The \definition{weight} of a diagram $D$, 
denoted as $\wt(D)$, 
is a weak composition 
whose $i$\textsuperscript{th}  entry is the number of boxes 
in its $i$\textsuperscript{th} row.
Each weak composition $\alpha$
is associated with a diagram $D(\alpha)$,
the unique left-justified diagram 
with weight $\alpha$.
Each permutation $w \in S_n$ or $S_+$
is associated with a diagram called 
the \definition{Rothe diagram} 
$$RD(w) := \{(r,c): w(r) > c, w(i) \neq c \textrm{ for any } i \in [r]\}.$$

\begin{exa}
We provide examples of two diagrams. 
For clarity, we put an ``$i$'' 
on the left of the $i$\textsuperscript{th} row
and put a small dot in each cell.
\begin{align*}
D((0,2,4,0,1)) = 
\raisebox{1cm}{
\begin{ytableau}
\none[1]\cr
\none[2] & \cdot & \cdot \cr
\none[3] & \cdot & \cdot & \cdot & \cdot \cr
\none[4] \cr
\none[3] & \cdot
\end{ytableau}}
\quad,\quad\quad\quad
RD([4,1,5,3,2]) = 
\raisebox{1cm}{
\begin{ytableau}
\none[1] & \cdot & \cdot & \cdot \cr
\none[2] & \none & \none & \none \cr
\none[3] & \none & \cdot & \cdot \cr
\none[4] & \none & \cdot & \none \cr
\none[5] & \none & \none & \none 
\end{ytableau}}\quad.
\end{align*}
\end{exa}

The Rothe diagram can characterize one special term in a Schubert polynomial.
We consider the \definition{tail-lexicographical order} on weak compositions: 
For two weak compositions $\alpha, \gamma$,
we say $\alpha$ is larger than $\gamma$ if there exists $i$
such that $\alpha_j = \gamma_j$ for all $j > i$
and $\alpha_i > \gamma_i$.
For a polynomial $f$,
the \definition{support} of $f$, denoted as $\supp(f)$,
is the set of weak composition $\alpha$
such that $x^\alpha$ has non-zero coefficient in $f$.
The \definition{leading monomial} of $f$ is $x^\alpha$
such that $\alpha$ is the largest in $\supp(\alpha)$.
By Lascoux and Sch{\"u}tzenberger~\cite{LS:Schubert},
the leading monomial of $\fS_w$ is $x^{\wt(RD(w))}$
with coefficient $1$.

To describe the leading monomial of a top Lascoux polynomial,
Pan and Yu~\cite{PY2} introduce the \definition{snow diagram}.
For each diagram $D$, its snow diagram $\snow(D)$
is a diagram together with some labels in its cells. 
Each cell can be unlabeled, or labeled by $\bullet$ or $\snowflake$.
We only consider the snow diagram of $D(\alpha)$
where $\alpha$ is a snowy weak composition. 
In this case, $\snow(D(\alpha))$ can be defined as follows. 
In $D(\alpha)$, 
label the rightmost cell on each row with $\bullet$.
Then put a cell labeled by $\snowflake$ in empty spaces
above each $\bullet$.
\begin{exa}
Let $\alpha = (2,0,4,0,1)$.
Then $D(\alpha)$ and $\snow(D(\alpha))$
are depicted as follows.

\begin{align*}
D(\alpha) = 
\raisebox{0.95cm}{
\begin{ytableau}
\none[1] & \cdot& \cdot \cr
\none[2] \cr
\none[3] & \cdot& \cdot& \cdot& \cdot \cr
\none[4] \cr
\none[5] & \cdot\cr
\end{ytableau}}\quad,
\quad \quad
\snow(D(\alpha)) = 
\raisebox{0.95cm}{
\begin{ytableau}
\none[1] & \cdot & \bullet & \none & \snowflake \cr
\none[2] & \snowflake & \none & \none & \snowflake\cr
\none[3] & \cdot& \cdot& \cdot& \bullet \cr
\none[4] & \snowflake \cr
\none[5] & \bullet \cr
\end{ytableau}}\quad.
\end{align*}
\end{exa}

For a snowy $\alpha$,
define 
$$\rajcode(\alpha) := \wt(\snow(D(\alpha))). \textrm{ Equivalently, }
\rajcode(\alpha)_i := \alpha_i + \{j > i: \alpha_j > \alpha_i \}.$$
By~\cite{PY2}, 
$x^{\rajcode(\alpha)}$ is the leading monomial
of $\topLas_\alpha$.
Moreover, two distinct snowy weak compositions
have different $\rajcode$.
We end this subsection with
a simple property of $\rajcode$
that will be useful in \S\ref{S: structure constants}.
\begin{lemma}
\label{L: rajcode tail-lex}
Let $\alpha, \gamma$ be two snowy weak compositions. 
If $\alpha$ is larger than $\gamma$ in tail-lexicographical order,
then $\rajcode(\alpha)$ is also larger than $\rajcode(\gamma)$.
\end{lemma}
\begin{proof}
Find $i$ such that $\alpha_j = \gamma_j$ for all $j > i$
and $\alpha_i > \gamma_i$.
Clearly, $\rajcode(\alpha)_j = \rajcode(\gamma)_j$ 
for all $j > i$ and $\rajcode(\alpha)_i > \rajcode(\gamma)_i$.
\end{proof}

\subsection{Dual character of the flagged Weyl module}
Let $B$ be the group of $n \times n$ upper triangular matrices
over $\mathbb{C}$.
Each diagram $D$ that lies in $[n] \times [n]$ 
is associated with a representation of $B$
known as the \definition{flagged Weyl module}.
See~\cite{Mac, RS_Percent} for a detailed construction of this module.

Let $D$ be a diagram.
We use $\chi_D$ to denote the dual character of the 
flagged Weyl module associated with $D$.
This is a family of polynomials 
in $\mathbb{Q}[x_1, x_2, \cdots]$
containing the Schubert polynomials and key polynomials. 

\begin{thm}[{\cite{Dem, KP, Mag}}]
\label{T: Schuberts are chars}
For $w \in S_n$,
$\chi_{RD(w)} = \fS_w$.
For a weak composition $\alpha$,
$\chi_{D(\alpha)} = \kappa_\alpha$.
\end{thm}

The polynomials $\chi_D$ can be computed recursively
for certain diagrams $D$.
To characterize them, we define the following moves
on diagrams. 
\begin{defn}[\cite{Mag}]
Let $D$ be a diagram.
We say $D'$ is obtained from $D$ via 
an \definition{orthodontic move}
if one of the following holds. 
\begin{itemize}
\item Suppose there is $r$
such that $(1,r') \in D$ if and only if $r' \in [r]$.
Then one may remove all cells in column $1$
of $D$ and shift all remaining cells
leftward by $1$ to obtain $D'$.
\item Assume there exists $r$ such that $(r, c) \in D$ implies $(r+1,c) \in D$
for any $c \in \mathbb{Z}_{>0}$.
Then one may swap row $r$ and row $r+1$ 
to obtain $D'$.
\end{itemize}
\end{defn}

The effect of each orthodontic move on $\chi_D$
can be characterized as follows. 
\begin{thm}[{\cite{Mag}}]
\label{T: ort}
Let $D$ be a diagram and $D'$ is obtained from $D$ via an orthodontic move.
\begin{itemize}
\item If $D'$ is obtained from $D$ by removing cells $(1, 1), \cdots, (r,c)$ and shifting all cells left,
then $\chi_{D} = \chi_{D'} x_1 \cdots x_r$.
\item If $D'$ is obtained from $D$ by swapping row $r$ and $r+1$,
then $\chi_{D} = \pi_i(\chi_{D'})$.
\end{itemize}
\end{thm}

Together with the base case $\chi_{\emptyset} = 1$,
one may compute $\chi_D$ using Theorem~\ref{T: ort}
if one can obtain the empty diagram from $D$
using the orthodontic moves. 
Such diagrams can be described as follows. 

\begin{defn}[\cite{RS_Percent}]
A diagram $D$ is \definition{percent-avoiding} if 
$(i_1,j_1), (i_2, j_2) \in D$ with $i_1 < i_2$ and $j_2 < j_1$
implies either $(i_1, j_2) \in D$ or $(i_2, j_1) \in D$. 
\end{defn}

\begin{thm}[\cite{RS_Percent}]
\label{T: PA and OM}
Let $D$ be a diagram. 
One may apply orthodontic moves on $D$ repetitively and
eventually obtain the empty diagram if and only if $D$
is percent-avoiding.
In particular, 
Rothe diagrams are such diagrams. 
\end{thm}

We also need a well-known property of $\chi_D$
that follows immediately from its construction.
\begin{thm}
\label{T: Permuting columns}
If $D$ and $D'$ differ by permuting the columns, 
then $\chi_D = \chi_{D'}$.
\end{thm}

\section{Relations between top Lascoux polynomials 
and Schubert polynomials}
\label{S: Relations}

This section describes the relationship between
top Lascoux and Schubert polynomials.

\subsection{The reverse complement involution on polynomials}

In this subsection, we describe 
a linear operator on polynomials. 
In the next subsection, we use this operator to transform
a top Lascoux polynomial into a Schubert polynomial.
We begin with an involution on certain weak compositions. 

\begin{defn}
Let $m, n$ be positive integers.
Define the \definition{reverse complement} operator $r_{m,n}$
on the set of weak compositions
$\alpha$ such that $\supp(\alpha) \subseteq [n]$
and $\alpha_i \leq m$ for all $i$.
We define 
\begin{align*}
r_{m,n}(\alpha) := (m - \alpha_n, \dots, m - \alpha_1).
\end{align*}
\end{defn}

Next, we analogously define $r_{m,n}$ on certain polynomials. 
\begin{defn}
Let $m, n$ be positive integers.
We extend $r_{m, n}$
to the set of polynomials in $x_1, \dots, x_n$
where the power of any $x_i$ is at most $m$.
We define it as the linear operator such that 
$r_{m,n}(x^\alpha) := x^{r_{m,n}(\alpha)}$.
Equivalently, we can define $r_{m,n}$ as
\begin{align*}
r_{m,n}(f) & := x_1^m \cdots x_n^m f(x_n^{-1}, \cdots, x_1^{-1}).
\end{align*}
\end{defn}

\begin{rem}
The operator $r_{m,n}$ on polynomials
is similar to the operator 
$$f \mapsto x_1^n \cdots x_n^n f(x_1^{-1}, \cdots, x_n^{-1})$$
considered by Huh, Matherne, M{\'e}sz{\'a}ros and St. Dizier~\cite{HMMS}.
In \cite[Theorem 6]{HMMS},
the authors apply this operator on a Schubert polynomial $\fS_w$
with $w \in S_n$ and show the resulting polynomial is Lorentzian
after normalization. 
Our $r_{m,n}$ is also similar to the operator
$$f \mapsto f(x_1^{-1}, \cdots, x_n^{-1})$$
which sends the character 
of a $\textrm{GL}_n$ module to the character of its dual. 
A more generalized operator $f \mapsto x^\alpha f(x_1^{-1}, \cdots, x_n^{-1})$ and its action on Schubert polynomials are studied by
Fan, Guo and Liu~\cite{FGL}.
\end{rem}

Next, we investigate how to swap $r_{m,n}$
with the operators: 
$\partial_i$, $\pi_i$, and $\widehat{\pi}_i$.

\begin{lemma}
\label{L: commute operators}

Suppose $r_{m,n}$ is defined on a polynomial $f$.
Take $i \in [n-1]$.
Clearly, $r_{m,n}$ is also defined on $\partial_i(f), \pi_i(f)$,
and $\widehat{\pi}_i(f)$.
Then we have
\begin{align*}
r_{m,n} (\partial_i(f)) &= \widehat{\pi}_{n-i} (r_{m,n}(f)),\\
r_{m,n} (\pi_i(f)) &= \pi_{n-i} (r_{m,n}(f)),\\
r_{m,n} (\widehat{\pi}_i(f)) &= \partial_{n-i} (r_{m,n}(f)).
\end{align*}
\end{lemma}

\begin{proof}
It is enough to assume $f = x^\alpha$,
which is a routine check.
\end{proof}

\subsection{Relating Schubert polynomials to top Lascoux polynomials}

In this subsection, we establish that each Schubert polynomial
is the reverse complement of a top Lascoux polynomial 
and vice versa.
We start by describing a variation of the map introduced by Fulton~\cite[(3.4)]{ful}.

\begin{defn}
Let $\alpha$ be a snowy weak composition.
Take any $m, n$ such that $\supp(\alpha) \subseteq [n]$
and $m \geq \max(\alpha)$.
The \definition{$(m,n)$-standardization} of $\alpha$,
denoted as $\std_{m,n}(\alpha)$ is the unique permutation 
$w$ satisfying $w(n+1) < w(n+2) < \cdots$ and 

\begin{align*}
w(i) := \begin{cases}
r_{m+1, n}(\alpha)_i & \text{if $r_{m+1, n}(\alpha)_i \leq m$} \\
m + |\{j \in [i]: r_{m+1, n}(\alpha)_j = m+1\}| &\text{if $r_{m+1, n}(\alpha)_i = m+1$,}
\end{cases}
\end{align*}
for any $i \in [n]$.
\end{defn}

For instance, 
if $\alpha = (2,4,0,6,0,0,1)$,
then $std_{6,7}(\alpha)$ has one-line notation
$[6,7,8,1,9,3,5,2,4]$.

Then we can describe the relation between
top Lascoux polynomials and Schubert polynomials.

\begin{thm}
\label{T: topLas as reverse Schubert}
Let $\alpha$ be a snowy weak composition.
Take any $m, n$ such that $\supp(\alpha) \subseteq [n]$
and $m \geq \max(\alpha)$.
Let $w$ be the permutation $\std_{m,n}(\alpha)$.
Then
$$
r_{m, n}(\topLas_{\alpha}) = \fS_w.
$$
\end{thm}

For instance, let $\alpha = (2,4,0,6,0,0,1)$,
$m = 6$ and $n = 7$.
Then $r_{6,7}(\topLas_\alpha) = \fS_{[6,7,8,1,9,3,5,2,4]}$.
\begin{proof}
Prove by induction on $\alpha$.
For the base case, assume $\alpha$ is a partition
with $\supp(\alpha) = [k]$.
Then $r_{m+1, n}(\alpha) = ( m+1, \cdots, m+1, m+1 - \alpha_k, \cdots, m+1 - \alpha_1)$.
The first $n$ numbers in the one-line notation of $w$ are
$$m+1, m+2, \cdots, m + n - k, m+1 - \alpha_k, \cdots, m+1 - \alpha_1.$$
Thus, we have $\invcode(w) = (m, \cdots, m, m - \alpha_k, \cdots, m - \alpha_1) = r_{m,n}(\alpha)$,
so $w$ is a dominant permutation.
By the definition of Schubert polynomials, 
$$\fS_w = x^{r_{m,n}(\alpha)} = r_{m,n}(x^\alpha) = \topLas_\alpha.$$

Now suppose $\alpha_i < \alpha_{i+1}$.
It is routine to check $w(n-i) < w(n-i+1)$ and 
$\std_{m,n}(s_i \alpha) = ws_{n-i}$.
Then the proof is finished by Lemma~\ref{L: commute operators}:
$$\fS_w = \partial_{n-i} (\fS_{w s_{n-i}}) =  \partial_{n-i} (\fS_{\std_{m,n}(s_i \alpha)}) = \partial_{n-i}(r_{m,n}(\topLas_{s_i \alpha})) = r_{m,n}(\widehat{\pi}_i(\topLas_{s_i \alpha})) = r_{m,n}(\topLas_\alpha).$$
\end{proof}

\begin{exa}
We can understand Theorem~\ref{T: topLas as reverse Schubert}
via a commutative diagram.
For instance, the equation $r_{4,5}(\topLas_{(2,0,4,0,1}) = \topLas_{[4,5,1,6,3,2]}$
is implied by the following commutative diagram.

\[
\begin{tikzcd}
\fS_{[5,6,4,3,1,2]} \arrow[r,"\partial_2"] \arrow[d,"r_{4,5}"] & 
\fS_{[5,4,6,3,1,2]} \arrow[r,"\partial_1"] \arrow[d,"r_{4,5}"] & 
\fS_{[4,5,6,3,1,2]} \arrow[r,"\partial_4"] \arrow[d,"r_{4,5}"] &
\fS_{[4,5,6,1,3,2]} \arrow[r,"\partial_3"] \arrow[d,"r_{4,5}"] & 
\fS_{[4,5,1,6,3,2]} \arrow[d,"r_{4,5}"]
\\
\topLas_{(4,2,1,0,0)} \arrow[r,"\widehat{\pi}_3"] \arrow[u] & 
\topLas_{(4,2,0,1,0)} \arrow[r,"\widehat{\pi}_4"] \arrow[u] & 
\topLas_{(4,2,0,0,1)} \arrow[r,"\widehat{\pi}_1"] \arrow[u] &
\topLas_{(2,4,0,0,1)} \arrow[r,"\widehat{\pi}_2"] \arrow[u] & 
\topLas_{(2,0,4,0,1)} \arrow[u]
\end{tikzcd}
\]
\end{exa}

Consequently, every Schubert polynomial is the reverse complement
of a top Lascoux polynomial. 
\begin{cor}
Consider $w \in S_n$.
Let $\alpha = (n+1 - w(n), \cdots, n+1 - w(2), n+1-w(1))$.
Then $\fS_w = r_{n,n}(\topLas_\alpha)$.
\end{cor}
\begin{proof}
Notice that $\std_{n,n}(\alpha) = w$.
Then the proof is finished by 
Theorem~\ref{T: topLas as reverse Schubert}.
\end{proof}

\subsection{A diagrammatic perspective of standardization}
In this subsection, 
we interpret the standardization map $\std_{m,n}$
in terms of diagrams. 
Recall that each snowy weak composition $\alpha$ is associated
with a labeled diagram $\snow(D(\alpha))$.
Each permutation $w$ is associated with the Rothe diagram $RD(w)$.
We describe the relationship between $\snow(D(\alpha))$
and $RD(\std_{m,n}(\alpha))$.

\begin{exa}
\label{E: compare diagrams}
Consider the snowy weak composition $\alpha = (0,4,2)$.
Let $m = 4$ and $n = 3$.
Let $w = \std_{m,n}(\alpha) = [3,1,5,2,4]$.
We depict $\snow(D(\alpha))$ and $RD(w)$ as follows: 
$$
\snow(D(\alpha)) = 
\raisebox{0.45cm}{
\begin{ytableau}
\none[1] & \none & \snowflake & \none & \snowflake \cr
\none[2] & \cdot & \cdot  & \cdot & \bullet\cr
\none[3] & \cdot & \bullet
\end{ytableau}}
\quad\quad
RD(w) = 
\raisebox{0.45cm}{
\begin{ytableau}
\none[1] & \cdot & \cdot \cr
\none[2] \cr
\none[3] & \none& \cdot & \none & \cdot
\end{ytableau}}
$$

We observe that both diagrams live in the first $3$ 
rows and $4$ columns. 
Imagine that we ignore the labels of $\snow(D(\alpha))$ 
and put it in an $3 \times 4$ box.
Then we take its complement within the box
and rotate the box by $180^\circ$.
What we get is exactly $RD(w)$. 

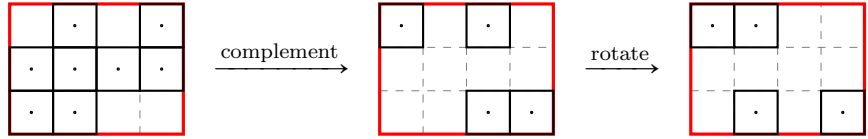
$$
\begin{tikzpicture}[x=1.5em,y=1.5em,thick,color = blue]
\draw[step=1,gray,ultra thin,dashed] (0,0) grid (4,3);
\draw[color=red, very thick, sharp corners] (0,0) rectangle (4,3);
\draw[color=black, thick, sharp corners] (0,1) rectangle (1,2);
\filldraw[black] (0.5,1.5) circle (0.2pt);
\draw[color=black, thick, sharp corners] (0,0) rectangle (1,1);
\filldraw[black] (0.5,0.5) circle (0.2pt);
\draw[color=black, thick, sharp corners] (1,2) rectangle (2,3);
\filldraw[black] (1.5,2.5) circle (0.2pt);
\draw[color=black, thick, sharp corners] (1,1) rectangle (2,2);
\filldraw[black] (1.5,1.5) circle (0.2pt);
\draw[color=black, thick, sharp corners] (1,0) rectangle (2,1);
\filldraw[black] (1.5,0.5) circle (0.2pt);
\draw[color=black, thick, sharp corners] (2,1) rectangle (3,2);
\filldraw[black] (2.5,1.5) circle (0.2pt);
\draw[color=black, thick, sharp corners] (3,2) rectangle (4,3);
\filldraw[black] (3.5,2.5) circle (0.2pt);
\draw[color=black, thick, sharp corners] (3,1) rectangle (4,2);
\filldraw[black] (3.5,1.5) circle (0.2pt);
\end{tikzpicture}
\quad \raisebox{0.8cm}{$\xrightarrow{\textrm{complement}}$} \quad 
\begin{tikzpicture}[x=1.5em,y=1.5em,thick,color = blue]
\draw[step=1,gray,ultra thin,dashed] (0,0) grid (4,3);
\draw[color=red, very thick, sharp corners] (0,0) rectangle (4,3);
\draw[color=black, thick, sharp corners] (0,2) rectangle (1,3);
\filldraw[black] (0.5,2.5) circle (0.2pt);
\draw[color=black, thick, sharp corners] (2,2) rectangle (3,3);
\filldraw[black] (2.5,2.5) circle (0.2pt);
\draw[color=black, thick, sharp corners] (2,0) rectangle (3,1);
\filldraw[black] (2.5,0.5) circle (0.2pt);
\draw[color=black, thick, sharp corners] (3,0) rectangle (4,1);
\filldraw[black] (3.5,0.5) circle (0.2pt);
\end{tikzpicture}
\quad \raisebox{0.8cm}{$\xrightarrow{\textrm{rotate}}$} \quad 
\begin{tikzpicture}[x=1.5em,y=1.5em,thick,color = blue]
\draw[step=1,gray,ultra thin,dashed] (0,0) grid (4,3);
\draw[color=red, very thick, sharp corners] (0,0) rectangle (4,3);
\draw[color=black, thick, sharp corners] (0,2) rectangle (1,3);
\filldraw[black] (0.5,2.5) circle (0.2pt);
\draw[color=black, thick, sharp corners] (1,2) rectangle (2,3);
\filldraw[black] (1.5,2.5) circle (0.2pt);
\draw[color=black, thick, sharp corners] (1,0) rectangle (2,1);
\filldraw[black] (1.5,0.5) circle (0.2pt);
\draw[color=black, thick, sharp corners] (3,0) rectangle (4,1);
\filldraw[black] (3.5,0.5) circle (0.2pt);
\end{tikzpicture}
$$
\end{exa}

To characterize this relation in general, 
we define the operator $r_{m,n}$ on certain diagrams.
\begin{defn}
Let $D$ be a diagram that lives in the first $n$
rows and first $m$ columns. 
Let $r_{m,n}(D)$ be the diagram obtained from the following process:
Place $D$ in a $n \times m$ box, take its complement within the box,
and rotate the box by $180^\circ$.
\end{defn}

\begin{lemma}
\label{L: compare diagram}
Let $\alpha$ be a snowy weak composition.
Take $n$ with $\supp(\alpha) \subseteq [n]$
and $m$ with $\max(\alpha) \leq m$.
Let $w = \std_{m,n}(\alpha)$.
Then $RD(w) = r_{m,n}(\snow(D(\alpha)))$.
\end{lemma}

\begin{proof}
If we ignore the labels, 
$\snow(D(\alpha))$ consists of cells $(r,c)$
such that $\alpha_r \geq c$
or $\alpha_i = c$ for some $i > r$.
Clearly, $\snow(D(\alpha))$ lies in the first $n$
rows and $m$ columns. 
If we complement $\snow(D(\alpha))$ 
within the $n \times m$ box,
the resulting diagram consists of $(r,c) \in [m] \times [n]$
such that 
$\alpha_r < c$ and 
$\alpha_i \neq c$ for any $i > r$.

Let $D' = r_{m,n}(\snow(D(\alpha)))$. 
We check $D'$ and $RD(w)$ agree row by row.
They clearly agree under row $n$. 
Consider $r \in [n]$.
If $\alpha_{n+1-r} = 0$,
then $D'$ has no cells in row $r$.
Also, $w(r) \geq m+1$.
By the definition of standardization map,
there are no inversions of $w$ of the form $(r, r')$.
Thus, $RD(w)$ also has no cells in row $r$.
If $\alpha_{n+1-r} > 0$,
$D'$ has cells in column $c$ for $c$
such that $c < m+1 - \alpha_{n+1-r}$
and $c \neq m+1 - \alpha_{n+1-i}$ for all $i \in [r]$.
In other words, that is all $c$ such that 
$c < w(r)$ and $c$ is not in $w(1), \cdots, w(r)$.
Therefore, $D'$ and $RD(w)$ agree on row $r$.

\end{proof} 

\section{Relations between top Lascoux structure constants and the Schubert structure constants}

\label{S: structure constants}

Recall that $\{\topLas_\alpha: \alpha \textrm{ is snowy}\}$
forms a basis of the algebra $\widehat{V}$.
The top Lascoux structure constant $d^{\delta}_{\alpha, \gamma}$
is the coefficient of $\topLas_\delta$
in the expansion of $\topLas_\alpha \topLas_\gamma$.
At this point, 
we do not have any reason to believe that 
they are positive integers. 
Surprisingly, the connection between top Lascoux
polynomials and Schubert polynomials
establishes a bridge between $d^{\delta}_{\alpha, \gamma}$
and the Schubert structure constants $c^w_{u,v}$. 
First, we describe a necessary condition
for $d^{\delta}_{\alpha, \gamma}$ to be non-zero.

\begin{lemma}
\label{L: condition for d nonzero}
Let $\alpha$, $\gamma$ and $\delta$ be snowy weak compositions. 
Find $m_1, m_2$ and $n$
such that $m_1 \geq \max(\alpha)$, $m_2 \geq \max(\gamma)$,
$\supp(\alpha) \subseteq [n]$ and $\supp(\gamma) \subseteq [n]$.
If $d^{\delta}_{\alpha, \gamma} \neq 0$,
we must have $\supp(\delta) \subseteq [n]$
and $\max(\delta) \leq m_1 + m_2$.
\end{lemma}
The proof relies heavily on the statistic $\rajcode$.

\begin{proof}
First, expand 
\begin{equation}
\label{EQ: product of top L}
\topLas_\alpha \times \topLas_\gamma = 
\sum_{\sigma} d^{\sigma}_{\alpha, \gamma} \topLas_\sigma.    
\end{equation}
We know the left hand side uses only variables $x_1, \dots, x_n$.
Moreover, in any monomial on the left hand side, 
each variable has power at most $m_1 + m_2$.
Let $S$ be the set of all $\sigma$ with $d^{\sigma}_{\alpha, \gamma} \neq 0$.
Among $S$, find the largest $\sigma$ in tail-lexicographical order.
By Lemma~\ref{L: rajcode tail-lex}, $\rajcode(\sigma)$ is also 
larger than $\rajcode(\sigma')$ for any $\sigma' \in S$.
Thus, $x^{\rajcode(\sigma)}$ has non-zero coefficient 
on the right hand side of (\ref{EQ: product of top L}),
so $\supp(\sigma) \subseteq [n]$.
It follows that $\supp(\sigma') \subseteq [n]$
for any $\sigma' \in S$.

Now find $\sigma \in S$ with the largest $\max(\sigma)$,
break ties by picking the largest in tail-lexicographical order.
Say $\max(\sigma) = m$.
In $\rajcode(\alpha)$, one entry is $m$.
We can see $\rajcode(\alpha)$ cannot appear in $\topLas_{\sigma'}$
for any other $\sigma' \in S$:
If so, then $\sigma'$ has an entry at least $m$
and $\rajcode(\sigma')$ is larger than $\rajcode(\sigma)$,
contradicting to the maximality of $\sigma$.
Thus, $x^\rajcode(\sigma)$ has non-zero coefficient 
on the right hand side of (\ref{EQ: product of top L}),
so $m \leq m_1 + m_2$.
It follows that $\max(\sigma') \leq m_1 + m_2$
for any $\sigma' \in S$.
\end{proof}

Now we describe the main theorem of this section.

\begin{thm}
\label{T: structure constants}
Let $\alpha$, $\gamma$ be snowy weak compositions. 
Find $m_1, m_2$ and $n$
such that $m_1 \geq \max(\alpha)$, $m_2 \geq \max(\gamma)$,
$\supp(\alpha) \subseteq [n]$ and $\supp(\gamma) \subseteq [n]$.
Let $u = \std_{m_1, n}(\alpha)$ and 
$v = \std_{m_2, n}(\gamma)$.
For any snowy weak composition $\delta$
with $\supp(\delta) \subseteq [n]$
and $\max(\delta) \leq m_1 + m_2$,
we let $w = \std_{m_1 + m_2, n}(\delta)$.
Then $d^{\delta}_{\alpha, \gamma} = c^w_{u,v}$.
\end{thm}

\begin{proof}
First, we have $\topLas_\alpha \times \topLas_\gamma = 
\sum_{\sigma} d^{\sigma}_{\alpha, \gamma} \topLas_\sigma$.
By Lemma~\ref{L: condition for d nonzero},
the sum is over all $\sigma$ with $\supp(\sigma) \subseteq [n]$
and $\max(\sigma) \leq m_1 + m_2$.
Apply $r_{m_1 + m_2, n}$ on both sides.
Using Theorem~\ref{T: topLas as reverse Schubert},
the left hand side becomes
\begin{align*}
r_{m_1 + m_2,n}(\topLas_\alpha \times \topLas_\gamma)
=r_{m_1,n}(\topLas_\alpha) r_{m_2,n}(\topLas_\gamma)
= \fS_u \fS_v.
\end{align*}
The right hand becomes 
$$
r_{m_1 + m_2,n}(\sum_{\sigma} d^{\sigma}_{\alpha, \gamma} \topLas_\sigma)
= \sum_{\sigma} d^{\sigma}_{\alpha, \gamma} r_{m_1 + m_2,n}(\topLas_\sigma)
= \sum_{\sigma} d^{\sigma}_{\alpha, \gamma} \fS_{\std_{m_1 + m_2, n}(\sigma)}.
$$

\end{proof}

\begin{exa}
\label{E: structure constants permutation}
Let $\alpha = (2,3,1,4)$ and $\gamma = (2,1,4,3)$.
We can let $m_1 = m_2 = n = 4$.
Then $u = \std_{4,4}(\alpha) = [1,4,2,3]$
and $v = \std_{4,4}(\gamma) = [2,1,4,3]$.
We compute
\begin{align*}
\topLas_{\alpha} \times \topLas_{\gamma} = 
\topLas_{(8,6,5,7)} + \topLas_{(6,8,4,7)} + \topLas_{(7,8,5,6)} 
+ \topLas_{(7,6,8,5)} + \topLas_{(6,7,8,4)}.\\
\fS_{u} \times \fS_{v} = \fS_{[2,4,3,1]} + \fS_{[2,5,1,3,4]} + \fS_{[3,4,1,2]} + \fS_{[4,1,3,2]} + \fS_{[5,1,2,3,4]}.      
\end{align*}

We check Theorem~\ref{T: structure constants} when $\delta = (8,6,5,7)$. 
We have $d^\delta_{\alpha, \gamma} = 1$.
Now we compute $w = \std_{4 + 4, 4}(\delta) = [2,4,3,1]$.
Indeed, $c^w_{u,v} = 1$.
\end{exa}

Theorem~\ref{T: structure constants} can express
each Schubert structure constant as a top Lascoux structure constant.
\begin{cor}
Take $u, v \in S_n$ and $w \in S_{2n}$.
Assume $w(n+1) < \cdots < w(2n)$.
Let $\alpha = (n+1 - u(n), \cdots, n+1-u(1))$,
$\gamma = (n+1 - v(n), \cdots, n+1-v(1))$
and $\delta = (n+1 - w(n), \cdots, n+1-w(1))$
Then $c^w_{u,v} = d^{\delta}_{\alpha, \gamma}$.
\end{cor}

\section{Bumpless pipedream formula}
\label{S: BPD}
In this section, we 
discuss another application of the relationship
between top Lascoux polynomials
and Schubert polynomials. 
Bumpless pipedreams (BPD) give a formula
to compute the monomial expansion 
of Schubert polynomials.
After ``reversing the BPDs'', we get a formula
for top Lascoux polynomials. 

\begin{defn}
Let $\alpha$ be an snowy weak composition. 
Find smallest $n$ such that $\supp(\alpha) \subseteq [n]$
and let $m = \max(\alpha)$.
A \definition{left-to-top BPD} of $\alpha$ 
is a grid with $n$ rows and $m$ columns 
built by tiles $\btile, \rtile, \jtile,
\ptile, \htile$ and $\vtile$.
For each $i \in [n]$ with $\alpha_i > 0$,
there is a pipe entering from the left
in row $i$ and 
goes to the top of column $\alpha_i$.
Moreover, no two pipes can cross more than once. 
Let $\LTBPD(\alpha)$ be the set of all 
left-to-top BPDs of $\alpha$.
\end{defn}

\begin{exa}
\label{E: BPD topL}
Consider $\alpha = (0,3,0,2)$.
Then $n = 4$ and $m = 3$.
The set $\LTBPD(\alpha)$ 
has five elements: 
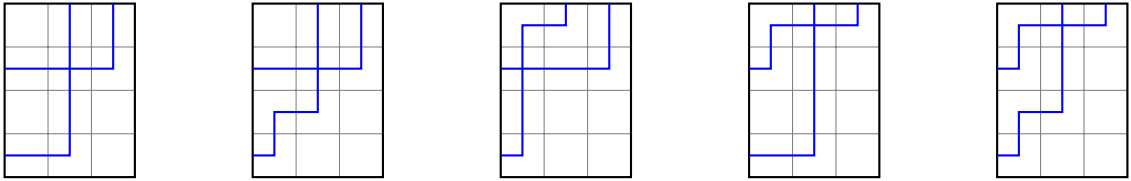
\[
\begin{tikzpicture}[x=1.5em,y=1.5em,thick,color = blue]
\draw[step=1,gray,thin] (0,0) grid (3,4);
\draw[color=black, thick, sharp corners] (0,0) rectangle (3,4);
\draw(0, 2.5)--(2.5,2.5)--(2.5,4);
\draw(0, 0.5)--(1.5,0.5)--(1.5,4);
\end{tikzpicture}
\quad\quad\quad\quad
\begin{tikzpicture}[x=1.5em,y=1.5em,thick,color = blue]
\draw[step=1,gray,thin] (0,0) grid (3,4);
\draw[color=black, thick, sharp corners] (0,0) rectangle (3,4);
\draw(0, 2.5)--(2.5,2.5)--(2.5,4);
\draw(0, 0.5)--(0.5,0.5)--(0.5,1.5)--(1.5,1.5)--(1.5,4);
\end{tikzpicture}
\quad\quad\quad\quad
\begin{tikzpicture}[x=1.5em,y=1.5em,thick,color = blue]
\draw[step=1,gray,thin] (0,0) grid (3,4);
\draw[color=black, thick, sharp corners] (0,0) rectangle (3,4);
\draw(0, 2.5)--(2.5,2.5)--(2.5,4);
\draw(0, 0.5)--(0.5,0.5)--(0.5,3.5)--(1.5,3.5)--(1.5,4);
\end{tikzpicture}
\quad\quad\quad\quad
\begin{tikzpicture}[x=1.5em,y=1.5em,thick,color = blue]
\draw[step=1,gray,thin] (0,0) grid (3,4);
\draw[color=black, thick, sharp corners] (0,0) rectangle (3,4);
\draw(0, 2.5)--(0.5,2.5)--(0.5,3.5)--(2.5,3.5)--(2.5,4);
\draw(0, 0.5)--(1.5,0.5)--(1.5,4);
\end{tikzpicture}
\quad\quad\quad\quad
\begin{tikzpicture}[x=1.5em,y=1.5em,thick,color = blue]
\draw[step=1,gray,thin] (0,0) grid (3,4);
\draw[color=black, thick, sharp corners] (0,0) rectangle (3,4);
\draw(0, 2.5)--(0.5,2.5)--(0.5,3.5)--(2.5,3.5)--(2.5,4);
\draw(0, 0.5)--(0.5,0.5)--(0.5,1.5)--(1.5,1.5)--(1.5,4);
\end{tikzpicture}
\]
\end{exa}

\begin{rem}

Readers might notice that left-to-top BPDs
look like the top left part of a BPD 
after rotation. 
Keep $\alpha$, $m$ and $n$ from Example~\ref{E: BPD topL}. 
Consider the classical BPDs of the permutation $\std_{m,n}(\alpha) = [2,4,1,5,3]$.
There are five of them:
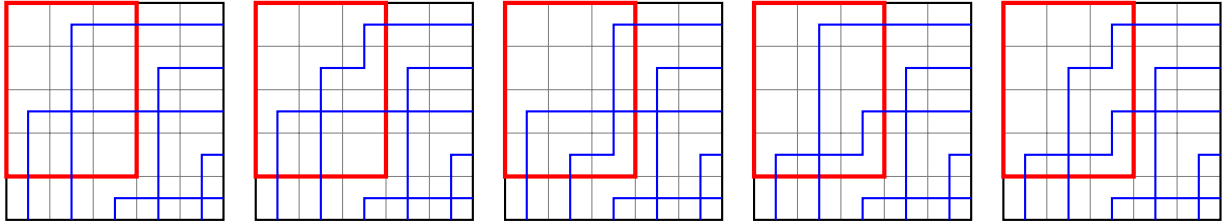
\[
\begin{tikzpicture}[x=1.5em,y=1.5em,thick,color = blue]
\draw[step=1,gray,thin] (0,0) grid (5,5);
\draw[color=black, thick, sharp corners] (0,0) rectangle (5,5);
\draw[color=red, ultra thick, sharp corners] (0,1) rectangle (3,5);
\draw(1.5, 0)--(1.5,4.5)--(5,4.5);
\draw(3.5, 0)--(3.5,3.5)--(5,3.5);
\draw(0.5, 0)--(0.5,2.5)--(5,2.5);
\draw(4.5, 0)--(4.5,1.5)--(5,1.5);
\draw(2.5, 0)--(2.5,0.5)--(5,0.5);
\end{tikzpicture}
\quad
\begin{tikzpicture}[x=1.5em,y=1.5em,thick,color = blue]
\draw[step=1,gray,thin] (0,0) grid (5,5);
\draw[color=black, thick, sharp corners] (0,0) rectangle (5,5);
\draw[color=red, ultra thick, sharp corners] (0,1) rectangle (3,5);
\draw(1.5, 0)--(1.5,3.5)--(2.5,3.5)--(2.5,4.5)--(5,4.5);
\draw(3.5, 0)--(3.5,3.5)--(5,3.5);
\draw(0.5, 0)--(0.5,2.5)--(5,2.5);
\draw(4.5, 0)--(4.5,1.5)--(5,1.5);
\draw(2.5, 0)--(2.5,0.5)--(5,0.5);
\end{tikzpicture}
\quad
\begin{tikzpicture}[x=1.5em,y=1.5em,thick,color = blue]
\draw[step=1,gray,thin] (0,0) grid (5,5);
\draw[color=black, thick, sharp corners] (0,0) rectangle (5,5);
\draw[color=red, ultra thick, sharp corners] (0,1) rectangle (3,5);
\draw(1.5, 0)--(1.5,1.5)--(2.5,1.5)--(2.5,4.5)--(5,4.5);
\draw(3.5, 0)--(3.5,3.5)--(5,3.5);
\draw(0.5, 0)--(0.5,2.5)--(5,2.5);
\draw(4.5, 0)--(4.5,1.5)--(5,1.5);
\draw(2.5, 0)--(2.5,0.5)--(5,0.5);
\end{tikzpicture}
\quad
\begin{tikzpicture}[x=1.5em,y=1.5em,thick,color = blue]
\draw[step=1,gray,thin] (0,0) grid (5,5);
\draw[color=black, thick, sharp corners] (0,0) rectangle (5,5);
\draw[color=red, ultra thick, sharp corners] (0,1) rectangle (3,5);
\draw(1.5, 0)--(1.5,4.5)--(5,4.5);
\draw(3.5, 0)--(3.5,3.5)--(5,3.5);
\draw(0.5, 0)--(0.5,1.5)--(2.5,1.5)--(2.5,2.5)--(5,2.5);
\draw(4.5, 0)--(4.5,1.5)--(5,1.5);
\draw(2.5, 0)--(2.5,0.5)--(5,0.5);
\end{tikzpicture}
\quad
\begin{tikzpicture}[x=1.5em,y=1.5em,thick,color = blue]
\draw[step=1,gray,thin] (0,0) grid (5,5);
\draw[color=black, thick, sharp corners] (0,0) rectangle (5,5);
\draw[color=red, ultra thick, sharp corners] (0,1) rectangle (3,5);
\draw(1.5, 0)--(1.5,3.5)--(2.5,3.5)--(2.5,4.5)--(5,4.5);
\draw(3.5, 0)--(3.5,3.5)--(5,3.5);
\draw(0.5, 0)--(0.5,1.5)--(2.5,1.5)--(2.5,2.5)--(5,2.5);
\draw(4.5, 0)--(4.5,1.5)--(5,1.5);
\draw(2.5, 0)--(2.5,0.5)--(5,0.5);
\end{tikzpicture}
\]

The top-to-left BPDs in Example~\ref{E: BPD topL}
are obtained by rotating the red part of these BPDs. 
\end{rem}

This pattern holds in general.

\begin{prop}
\label{P: rotate BPD}
Let $\alpha$ be a snowy weak composition. 
Find smallest $n$ such that $\supp(\alpha) \in [n]$
and let $m = \max(\alpha)$.
Let $w = \std_{m,n}(\alpha)$.
Then $\LTBPD(\alpha)$
is formed by rotating the first $n$ rows, $m$ columns
of BPDs in $\BPD(w)$.
\end{prop}
\begin{proof}
Immediate from Lemma~\ref{L: compare diagram}.
\end{proof}

Now we are ready to introduce a combinatorial formula
for the top Lascoux polynomials. 
We just need to specify the ``weight'' of a left-to-top BPD.

\begin{defn}
Let $D$ be a left-to-top BPD for some snowy weak composition.
The \definition{non-blank weight} of $D$,
denoted as $\wt_{\boxdot}(D)$, 
is a weak composition where the $i^\textsuperscript{th}$
entry counts the number of \textbf{non-blank tiles} in row $i$ of $D$.
\end{defn}

\begin{thm}
\label{T: BPD topL}
Let $\alpha$ be a snowy weak composition. 
Then
$$
\topLas_\alpha = \sum_{D \in \LTBPD(\alpha)} \wt_{\boxdot}(D).
$$
\end{thm}

For instance, Example~\ref{E: BPD topL} yields
$$
\topLas_{(0,3,0,2)} = x_1^2x_2^3x_3x_4^2 + x_1^2 x_2^3x_3^2x_4
+ x_1^3 x_2^3 x_3x_4 + x_1^3x_2^2x_3x_4^2 + x_1^3 x_2^2 x_3^2 x_4.
$$

\begin{proof}
Follows from Theorem~\ref{T: topLas as reverse Schubert}
and Proposition~\ref{P: rotate BPD}.
\end{proof}

\section{Key expansion of top Lascoux polynomials}
\label{S: key expansion}
In this section, we expand 
top Lascoux polynomials positively into key polynomials.
Our main tool is the Schubert-to-key expansion:
\begin{thm}[\cite{RS}]
For $w \in S_n$,
the Schubert polynomial $\fS_w$
can be expanded as $\sum_{\alpha} c^w_\alpha \kappa_\alpha$,
where the coefficients $c^w_\alpha$ are non-negative integers
counting certain tableaux. 
\end{thm}

We will translate this expansion to top Lascoux polynomials.
The first step is to understand how the $r_{m,n}$ operator
affects a key polynomial.
\begin{prop}
\label{P: reverse key}
Let $\alpha$ be any weak composition.
Let $m, n$ be positive integers.  
Then $r_{m,n}$ is defined on $\alpha$ if and only if
$r_{m,n}$ is defined on $\kappa_\alpha$.
If this is the case, 
$r_{m,n}(\kappa_\alpha) = \kappa_{r_{m,n}(\alpha)}$.
\end{prop}
The proof is essentially the same as the proof
of Theorem~\ref{T: topLas as reverse Schubert}.

\begin{proof}
The first claim is immediate. 
We prove the equation by induction on $\alpha$.
For the base case, assume $\alpha$ is a partition.
So is $r_{m,n} (\alpha)$.
We have 
$r_{m,n}(\kappa_\alpha) = r_{m,n}(x^\alpha) 
= x^{r_{m,n}(\alpha)} = \kappa_{r_{m,n}(\alpha)}$.

Now assume $\alpha_i < \alpha_{i+1}$ for some $i$.
For our inductive hypothesis, 
assume $\kappa_{s_i \alpha} = r_{m,n}(\kappa_{r_{m,n}(s_i \alpha)})$.
By Lemma~\ref{L: commute operators},
$$
\kappa_\alpha = \pi_i(\kappa_{s_i \alpha})
= \pi_i (r_{m,n} (\kappa_{r_{m,n}(s_i \alpha)}))
= r_{m,n} (\pi_{n-i} (\kappa_{s_{n-i} (r_{m,n}(\alpha)})))
= r_{m,n} (\kappa_{r_{m,n}(\alpha)}).
$$

\end{proof}

\begin{cor}
Let $\alpha$ be a snowy weak composition.
Find $m, n$ such that $\alpha \subseteq [n]$
and $\max(\alpha) \leq m$.
Let $w$ be the permutation $\std_{m,n}(\alpha)$.
The top Lascoux polynomial $\topLas_\alpha$
can be expanded as $\sum_{\gamma} c^w_\gamma \kappa_{r_{m,n}(\gamma)}$ where the coefficient 
$c^w_\gamma$ is the coefficient
of $\kappa_\gamma$ in the expansion of $\fS_w$.   
\end{cor}
\begin{proof}
By Theorem~\ref{T: topLas as reverse Schubert},
we know $\topLas_{\alpha} = r_{m,n}(\fS_w)$.
By the Schubert-to-key expansion and Proposition~\ref{P: reverse key},
$$
\topLas_{\alpha} = r_{m,n}(\sum_{\gamma} c^w_\gamma \kappa_{\gamma})
= \sum_{\gamma} c^w_\gamma r_{m,n}(\kappa_{\gamma})
= \sum_{\gamma} c^w_\gamma \kappa_{r_{m,n}(\gamma)}.
$$
\end{proof}

\section{Top Lascoux polynomials as dual-characters of flagged Weyl modules}
\label{S: character}

In this section, we show $\topLas_\alpha$ 
agrees with $\chi_{\snow(D(\alpha))}$.
Our proof relies on the following relation between 
the $r_{m,n}$ on polynomials and the $r_{m,n}$ on diagrams. 
\begin{prop}
\label{P: chi rotate}
Let $D$ be a diagram in the first $n$ rows and $m$ columns. 
Assume one can obtain the empty diagram 
by applying orthodontic moves from $D$. 
Then 
\begin{align}
\label{EQ: chi reverse}
\chi_{r_{m,n}(D)} = r_{m,n}(\chi_D).
\end{align}
\end{prop}
\begin{proof}
We prove by induction on $m$.
The case when $m = 0$ is trivial.

Suppose $m > 0$.
If $D$ is the empty diagram,
then $r_{m,n}(D)$ is the $n \times m$ box. 
We have $\chi_{r_{m,n}(D)} = x_1^m \cdots x_n^m$,
which agrees with $r_{m,n}(\chi_D) = r_{m,n}(1)$.

Now suppose $D$ is not empty and 
let $D'$ be a diagram obtained from $D$ after one
orthodontic move.
It remains to check~(\ref{EQ: chi reverse})
while assuming $\chi_{r_{m,n}(D')} = r_{m,n}(\chi_{D'})$.
We consider the two types of orthodontic moves.
\begin{itemize}
\item Suppose $D'$ is obtained from $D$
by removing cells $(1,1), \cdots, (1,r)$
and shift all remaining cells leftward. 
Thus, $D'$ lives in the $n \times (m-1)$ 
box. 
The diagram $r_{m,n}(D)$ is obtained by 
adding cells $(1,m), \cdots, (n-r, m)$ 
to $r_{m-1, n}(D')$.
Let $D''$ be the diagram obtained by moving 
column $m$ of $r_{m,n}(D)$ to column $1$
and shift all other column rightward by $1$.
By Theorem~\ref{T: Permuting columns},
$\chi_{r_{m,n}(D)} = \chi_{D''}$.
Notice that $r_{m-1,n}(D')$ is obtained from
$D''$ via an orthodontic move,
so $\chi_{D''} = x_1\cdots x_{n-r}\chi_{r_{m-1,n}(D')}$.
We have
$$
r_{m,n}(\chi_D) = r_{m,n}(x_1 \cdots x_r \chi_{D'})
= x_1\cdots x_{n-r} r_{m-1,n}(\chi_{D'})
=x_1\cdots x_{n-r} \chi_{r_{m-1,n}(D')} = \chi_{r_{m,n}(D)}.
$$

\item Suppose $D'$ is obtained from $D$
by swapping row $r$ and row $r+1$.
Then $r_{m,n}(D')$ is obtained from $r_{m,n}(D)$
via an orthodontic move that swaps row $n-r$
and row $n-r+1$.
By Lemma~\ref{L: commute operators}, 
we have 
$$
r_{m,n}(\chi_D) = r_{m,n}(\pi_r(\chi_{D'}))
= \pi_{n-r}(r_{m,n}(\chi_{D'}))
= \chi_{r_{m,n}(D)}. \qedhere
$$
\end{itemize}
\end{proof}

\begin{rem}
Fix an $n \in \mathbb{Z}_{>0}$.
The symmetrization operator $\pi_{w_0}$
can be defined as 
$$
\pi_{w_0} := (\pi_1 \cdots \pi_{n-1})
(\pi_1 \cdots \pi_{n-2}) \cdots (\pi_1 \pi_2) (\pi_1).
$$
It turns a polynomial in $x_1, \cdots, x_n$
into a polynomial symmetric in these variables.
Magyar~\cite{Mag} showed that for any 
diagram $D$ living in the first $n$ rows and $m$ columns,
$$
\pi_{w_0}(\chi_{r_{m,n}(D)}) = r_{m,n}(\pi_{w_0}(\chi_D)).
$$
Our statement in Proposition~\ref{P: chi rotate}
can be viewed as the ``non-symmetric refinement''
of Magyar's result.
\end{rem}

\begin{cor}
\label{C: top Las char}
Let $\alpha$ be a snowy weak composition.
We have $\topLas_\alpha = \chi_{\snow(D(\alpha))}$.
\end{cor}
\begin{proof}
Take $n$ with $\supp(\alpha) \subseteq [n]$
and $m$ with $\max(\alpha) \leq m$.
Let $w = \std_{m,n}(\alpha)$. 
By Lemma~\ref{L: compare diagram},
$RD(w) = r_{m,n}(\snow(D(\alpha))$.
By Theorem~\ref{T: PA and OM} and Proposition~\ref{P: chi rotate},
$r_{m,n}(\chi_{RD(w)}) = \chi_{\snow(D(\alpha))}$.
On the other hand, 
by Theorem~\ref{T: Schuberts are chars},
$\chi_{RD(w)} = \fS_w$.
Thus, 
$$
\topLas_\alpha 
= r_{m,n}(\fS_w) = r_{m,n}(\chi_{RD(w)}) = \chi_{\snow(D(\alpha))}. \qedhere
$$
\end{proof}

\begin{rem}
\label{R: top Gro char}
In~\cite{HMSS}, the authors focused on Grothendieck polynomials
indexed by vexillary (i.e. $2143$-avoiding) permutations.
For $w$ vexillary, they constructed a diagram $D^{top}(w)$
and conjectured $\topGro_w$ is a scalar multiple of $\chi_{D^{top}(w)}$.
By Pechenik and Scrimshaw~\cite{PS},
in this case, 
$\fG_w$ is a Lascoux polynomial.
Then by~\cite{PY2},
$\topGro_w$ is a scalar multiple 
of $\topLas_\alpha$ for some snowy $\alpha$.
As a conclusion, Corollary~\ref{C: top Las char} reduces their conjecture
to a purely combinatorial 
problem: check the two diagrams 
$D^{top}(w)$ and $\snow(D(\alpha))$
differ by permutation of columns. 
\end{rem}

We conclude this section by deriving a few properties
related to the support of $\topLas_\alpha$,
showing that its support has similar properties 
as the Schubert polynomials. 
Adve, Robichaux, and Yong~\cite{ARY}
characterized the support of Schubert polynomials 
using certain fillings:

\begin{defn}\cite{ARY}
Let $D$ be a diagram and $\alpha$ be a weak composition.
Then $\PT(D,\alpha)$ is the set of fillings of $D$ satisfying all of the following:
\begin{itemize}
\item For each $k$, the number of cells in $D$ filled by $k$ is $\alpha_k$.
\item In each column, numbers are increasing from top to bottom. 
\item Any entry in row $i$ is at most $i$.
\end{itemize}
\end{defn}

\begin{exa}\cite[Example 1.4]{ARY}
\label{E: PT Rothe}
Consider $D = RD((3,1,5,2,4))$.
We enumerate the six elements in
$\bigcup_\alpha \PT(D, \alpha)$:
$$
\begin{ytableau}
\none[1] & 1 & 1 \cr
\none[2] \cr
\none[3] & \none& 2 & \none & 1
\end{ytableau}
\quad\quad\quad\quad
\begin{ytableau}
\none[1] & 1 & 1 \cr
\none[2] \cr
\none[3] & \none& 2 & \none & 2
\end{ytableau}
\quad\quad\quad\quad
\begin{ytableau}
\none[1] & 1 & 1 \cr
\none[2] \cr
\none[3] & \none& 2 & \none & 3\cr
\none
\end{ytableau}
$$

$$
\begin{ytableau}
\none[1] & 1 & 1 \cr
\none[2] \cr
\none[3] & \none& 3 & \none & 1
\end{ytableau}
\quad\quad\quad\quad
\begin{ytableau}
\none[1] & 1 & 1 \cr
\none[2] \cr
\none[3] & \none& 3 & \none & 2
\end{ytableau}
\quad\quad\quad\quad
\begin{ytableau}
\none[1] & 1 & 1 \cr
\none[2] \cr
\none[3] & \none& 3 & \none & 3
\end{ytableau}
$$
\end{exa}

\begin{thm}\cite[Theorem 1.3]{ARY}
For a permutation $w$,
the support of $\fS_w$ is the set of $\alpha$
such that $\PT(RD(w), \alpha) \neq \emptyset$.
\end{thm}

For instance, Example~\ref{E: PT Rothe} 
says that $\supp(\fS_{(3,1,5,2,4)})$
consists of $(3,1), (2,2), (2,1,1), (3,0,1)$
and $(2,0,2)$.
Fink, M{\'e}sz{\'a}ros, and St. Dizier~\cite{FMS}
extended this characterization to $\chi_D$.
\begin{thm}\cite[Theorem 7]{FMS}
For a diagram $D$,
$\chi_D$ has SNP and its support $\chi_D$ 
is the set of $\alpha$
such that $\PT(D, \alpha) \neq \emptyset$.
\end{thm}
Consequently, the support of a top Lascoux polynomial 
can be characterized in the same manner.
\begin{cor}
\label{C: support of topL}
For a snowy weak composition $\alpha$,
$\topLas_\alpha$ has saturated Newton polytope. 
View $\snow(D(\alpha))$ as a diagram with labels erased. 
The support of $\topLas_\alpha$ 
is the set of $\gamma$ such that \\
$\PT(\snow(D(\alpha)), \gamma) \neq \emptyset$.
\end{cor}

\begin{exa}
\label{E: PT snow}
Consider the snowy weak composition $(0,4,2)$.
We enumerate the six elements in
$\bigcup_\alpha \PT(\snow(D(\alpha)), \alpha)$:
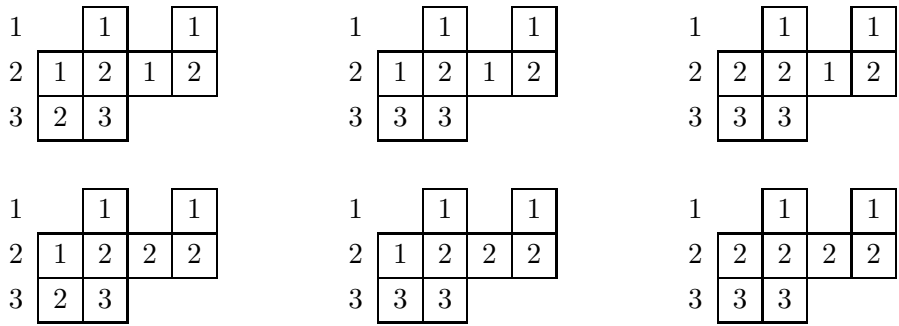
$$
\begin{ytableau}
\none[1] & \none & 1 & \none & 1 \cr
\none[2] & 1 & 2  & 1 & 2\cr
\none[3] & 2 & 3
\end{ytableau}
\quad\quad\quad\quad
\begin{ytableau}
\none[1] & \none & 1 & \none & 1 \cr
\none[2] & 1 & 2  & 1 & 2\cr
\none[3] & 3 & 3
\end{ytableau}
\quad\quad\quad\quad
\begin{ytableau}
\none[1] & \none & 1 & \none & 1 \cr
\none[2] & 2 & 2  & 1 & 2\cr
\none[3] & 3 & 3
\end{ytableau}
$$

$$
\begin{ytableau}
\none[1] & \none & 1 & \none & 1 \cr
\none[2] & 1 & 2  & 2 & 2\cr
\none[3] & 2 & 3
\end{ytableau}
\quad\quad\quad\quad
\begin{ytableau}
\none[1] & \none & 1 & \none & 1 \cr
\none[2] & 1 & 2  & 2 & 2\cr
\none[3] & 3 & 3
\end{ytableau}
\quad\quad\quad\quad
\begin{ytableau}
\none[1] & \none & 1 & \none & 1 \cr
\none[2] & 2 & 2  & 2 & 2\cr
\none[3] & 3 & 3
\end{ytableau}
$$

By Proposition~\ref{C: support of topL},
the support of $\topLas_{(0,4,2)}$
consists of $(4,3,1), (4,2,2), (3,3,2), (3,4,1)$ and $(2,4,2)$.
\end{exa}

\section{Acknowledgments}
We are grateful to Elena Hafner, Jianping Pan, 
Brendon Rhoades, Travis Scrimshaw and Avery St. Dizier
for carefully reading an earlier
version of this paper 
and giving many useful comments.
We also thank Sara Billey,
Zachary Hamaker, Daoji Huang, Oliver Pechenik, 
Mark Shimozono, 
and Alexander Yong
for helpful conversations. 
\bibliographystyle{alpha}
\bibliography{main}{}
\end{document}